\documentclass[a4paper,10pt,fleqn]{amsart}

\usepackage[utf8]{inputenc}
\usepackage{amsmath}
\usepackage{amsfonts}
\usepackage{amsthm}
\usepackage{amssymb}
\usepackage{amsaddr}
\usepackage{bbm}

\newtheorem{theorem}{Theorem}[section]
\newtheorem{lemma}[theorem]{Lemma}
\newtheorem{corollary}[theorem]{Corollary}
\newtheorem{proposition}[theorem]{Proposition}

\theoremstyle{remark}
\newtheorem{remark}[theorem]{Remark}

\theoremstyle{definition}
\newtheorem{definition}[theorem]{Definition}

\DeclareMathOperator{\Bor}{Bor}

\DeclareMathOperator{\Iso}{Iso}
\DeclareMathOperator{\JIso}{J-Iso}
\DeclareMathOperator{\Spec}{Spec}

\DeclareMathOperator{\sgn}{sgn}
\DeclareMathOperator*{\esssup}{ess\,sup}

\newcommand{\dt} {\partial_t}

\newcommand{\grad} {\nabla_{\!z}}
\newcommand{\X} {{\mathbb{R}^d}}
\newcommand{\R} {\mathbb{R}}
\newcommand{\N} {\mathbb{N}}
\newcommand{\T} {[0,\infty)}
\newcommand{\FT}[1] {\mathcal{F}\big\{ #1 \big\}}
\renewcommand{\phi} {\varphi}

\renewcommand{\=} {\overset{d}{=}}
\newcommand{\viii} {|\kern-0.25ex|\kern-0.25ex|}
\newcommand{\K}[2] {\viii #1 \viii_{K,#2}}

\newcommand{\Lp} {{L^p(\Omega)}}
\newcommand{\Ld} {{L^2(\Omega)}}
\newcommand{\fLap} {(-\Delta)^s}

\begin{document}
\title[]{Convection-Diffusion Equations with Random Initial Conditions}
\author{Mi\l{}osz Krupski}
\address{Uniwersytet Wroc\l{}awski, Instytut Matematyczny\\ pl. Grunwaldzki 2/4, 50-384 Wroc\l{}aw}
\email{milosz.krupski@math.uni.wroc.pl}
 
% \ccode{35R60, 60G10, 60G60, 35B45, 35A01, 35K61}
% \keywords{non-local Burgers equation; random initial conditions; homogeneous random fields}

  \begin{abstract}
We consider an evolution equation generalising the viscous Burgers equation
supplemented by an initial condition which is a homogeneous random field.
We develop a non-linear framework enabling us to show the existence
and regularity of solutions as well as their long time behaviour.
  \end{abstract}
\maketitle
  \begin{section}{Introduction}
   The main objective of this paper is to investigate solutions to a non-local
   analogue of the viscous Burgers equation with random initial conditions
      \begin{equation}\label{intro}
      \left\{
      \begin{aligned}
        &\dt u + \fLap u = \grad f(u)\quad&\text{on $\T\times\X$}, \\
        &u(0) \= u_0\quad&\text{on $\X$}.
      \end{aligned}
      \right.
    \end{equation}
   Here the operator $-\fLap$ denotes the (now) standard fractional Laplace operator
   with $s\in(\frac{1}{2},1]$
   and the initial condition $u_0$ is a real, homogeneous, isotropic random field (as defined in Section~\ref{preliminaries})
   whose finite-order moments are all bounded.
   By $\grad$ we denote the directional derivative. The function $f$ is a smooth function of polynomial growth.
   
   Such equations have been studied thoroughly in the deterministic case. In papers~\cite{MR1708995,MR1881259,MR1849690} the authors consider
   initial conditions which are integrable functions and describe certain properties of solutions that resonate
   with some of the results presented here. However, in the context of random initial data new
   methods had to be developed. Important results were obtained for bounded (deterministic) initial conditions~\cite{MR2019032}
   and we rely on them where possible (see Lemmas~\ref{classical-solutions}, \ref{classical-solutions-equiv}).
   Interesting developments were recently described in \cite{2017arXiv170302908I}.
   
   In general, partial differential equations with random initial data (homogeneous, stationary, isotropic, etc.) have been
   examined before, notably to describe certain physical phenomena, such
   as the Large Scale Structure of the Universe~\cite{MR1305783, MR1732301, MannJr.2001}.
   
   More specifically, there are numerous results available on the local Burgers equation (i.e. equation~\eqref{intro} with $s=1$ and $f(u)=u^2$) or very similar equations with random initial data
   \cite{MR1859007,MR1939652,MR1978661,MR1687092,MR1642664}. Equations of this type, however, have a curious property, exploited by the Hopf-Cole transformation, allowing them to be treated essentially as linear problems. Such a reduction is not known to be possible
   both in the non-local setting (i.e. with the fractional Laplacian) and for a more general
   function $f$, 
   which forces us to conduct a more involved non-linear analysis of the problem. 
   
   A significant part of this analysis are a~priori estimates, in the context of random initial data. Early results
   were already obtained by Rosenblatt~\cite{MR0264252} (in the local setting).
   Taking them as an inspiration we obtained new and much more general estimates.
   In fact, in Theorem~\ref{rosenblatt-lp} we prove that for a solution $u(t)$
   to problem~\eqref{intro} we have $E|u(t)|^p \leq E|u_0|^p$ for every $t\geq 0$.
   
   One difficulty when working with random fields is the question of regularity of its individual realisations.
   As it turns out, it is impossible to directly apply the classical theory ``pathwise'', treating $x\mapsto u_0(x,\omega)$
   as an initial condition of the non-random problem for every $\omega\in\Omega$ separately (cf. Proposition~\ref{iso-infinity} and Remark~\ref{l-infinity}).
   
   On the other hand, restricting the problem to admit only homogeneous random fields has a technical
   advantage. 
   By calculating the expected value $E$ not only we ``eliminate'' the variable $\omega$,
   as is normally the case, but also $x$ (because the field has identical distribution
   at every point in space the result is constant, see Remark~\ref{independent}).
   However simple, this property is essential to obtain e.g. the analogue of the Stroock-Varopoulos 
   inequality (see \cite{MR1218884,MR1409835}) in Lemma~\ref{cut-off-reg}.
   Another important observation is expressed by Lemma~\ref{derivative-zero} and we invite the readers to turn their attention to it.
   
   The remainder of the paper is structured as follows. After introducing the notation and basic concepts in Section~\ref{preliminaries}, in Section~\ref{linear-equation} we review the 
   results on the linear equation, which is the starting point for the rest of the theory. In Section~\ref{sec-lipschitz} we construct solutions in the case when the function $f$ is assumed to be Lipschitz. In Section~\ref{apriori} we 
   are able to establish the a~priori estimates. 
   Finally, in Section~\ref{general} we study the problem for functions $f$ of polynomial growth.
   The main result on the existence of solutions to problem~\eqref{intro} is expressed
   in Theorem~\ref{main}.
   
   This paper is a continuation of previous work~\cite{MR3628179} which dealt with linear problems.
  \end{section}
  \begin{section}{Isometry-invariant random fields}\label{preliminaries}
    \begin{subsection}{Basic notation}
      We denote the Borel sigma-algebra on $\X$ by $\Bor(\X)$ and the Lebesgue measure by $dx$.
      We use the Fourier transform defined as
      \begin{equation*}
        \FT{f}(\xi) = \int_\X e^{i\xi\cdot x}f(x)\,dx.
      \end{equation*}
      
      Given a measure space $(X,\Theta,\mu)$, by $L^p(X,\Theta,\mu)$ we denote the space of all $\Theta$-measurable real functions such that the norm defined in the case $1\leq p < \infty$ as
      \begin{equation*}
        \|f(x)\|_{L^p(X,\Theta,\mu)}^p = \int_X |f|^p d\mu = \int_X |f(x)|^p \mu(dx),
      \end{equation*}
      or in the case of $p=\infty$ as the value 
      \begin{equation*}
        \|f(x)\|_{L^\infty(X,\Theta,\mu)} = \esssup_{x\in X}|f| = \inf\big\{y\in X: \mu\{x\in X:|f(x)|>y\} = 0\big\},
      \end{equation*}
      is finite. Usually we shorten the notation and write $L^p(X,\mu)$.
      
      Let $(X,\|\cdot\|_X)$ be a subset of a normed space. For every fixed $K\geq0$ we may define the Bielecki norm
  \begin{equation}\label{bielecki}
      \K{u}{X} = \sup_{t\geq0}e^{-tK}\|u(t)\|_X
   \end{equation}
  and the space
  \begin{equation*}
   \mathcal{B}_{K,X} = \big\{u\in C(\T,X) : \K{u}{X} < \infty\big\}.
  \end{equation*}
      
      Let us fix a probability space $(\Omega,\Sigma,P)$ and denote $\Lp = L^p(\Omega,\Sigma,P)$.
      The particular case of $p=2$ constitutes a
      Hilbert space with the standard inner product $EXY$ defined for all $X,Y\in \Ld$.
      We write
      \begin{equation*}
        (X_1,\ldots,X_k) \= (Y_1,\ldots,Y_k),\quad\text{where $X_i,Y_i\in \Lp$ and $k\in\N$},
      \end{equation*}
      if both random vectors have the same probability distributions.
    \end{subsection}
\begin{subsection}{Random fields}
  Consider $\Lp$-continuous and $\Lp$-bounded random field $u$ on $\X$, i.e. $u\in C_b(\X,\Lp)$.
  We endow this space with the norm
  \begin{equation*}
    \|u\|_p = \sup_{x\in\X}\|u(x)\|_{\Lp}.
  \end{equation*}
  For $u,v\in C_b(\X,\Lp)$ we say that $u\=v$ if
  \begin{equation*}
    \big(u(x_1),\ldots,u(x_n)\big) \= \big(v(x_1),\ldots,v(x_n)\big)
  \end{equation*}
  as random vectors for every finite collection of points $x_1,\ldots,x_n\in \X$.
  \begin{definition}
  Let $1\leq p \leq \infty$. Consider the group $\Phi$ of isometries on $\X$ and a random field $u\in C_b(\X,\Lp)$.
  For a given function $\phi\in\Phi$ let $(\phi u)(x) =  u(\phi(x))$.
  We define the space $\Iso_p$ of \emph{isometry-invariant random fields} with finite $p$-th moment as
  \begin{equation*}
      \Iso_p = \{u\in C_b(\X,\Lp): u \= \phi u \text{ for every $\phi\in\Phi$}\}.
  \end{equation*}
  \end{definition}
  \begin{remark}
  Notice that the space $\Iso_p$ inherits the norm (and hence topology) of the space $C_b(\X,\Lp)$
  and it forms a closed subspace therein. Indeed, fix an isometry $\phi\in\Phi$ and a sequence $u_n\in\Iso_p$
  such that $\lim_{n\to\infty}\|u_n-u\|_p = 0$.
  Then $u_n\=\phi u_n$ and because
  \begin{equation*}
   \|\phi u_n-\phi u\|_p= \sup_{x\in\X}\|u_n(\phi(x))-u(\phi(x))\|_\Lp= \|u_n-u\|_p
  \end{equation*}
we have $u \= \phi u$, hence $u\in\Iso_p$.
  This also entails completeness.
  \end{remark}
  \begin{remark}\label{independent}
  Notice that the property $u\= \phi u$ implies that neither $\|u(x)\|_p$ nor $Eu(x)$ depend on the $x$ variable. 
  \end{remark}
  \begin{remark}
  The space $\Iso_p$ is not linear. For a simple example, let $X,Y$ be two independent random variables with identical distribution $\mathcal{N}(0,1)$.
  Define random fields $u,v\in C_b(\R,\Ld)$ as $u(x) = X$ and $v(x) = \sin(x)X + \cos(x)Y$.
  Notice that $u,v\in\Iso_2$. However, by calculating 
  \begin{equation*}
  E(u(x)+v(x))^2 = 2 \sin(x)+2
  \end{equation*}
  we discover that $u+v\notin\Iso_2$ (cf. Remark~\ref{independent}).
 \end{remark}
  In the particular case of $p=2$, for every random field $u\in\Iso_2$ we have $Eu(x)u(y) = B(|x-y|)$ for a certain positive definite function $B\in C(\R,\R)$. Moreover we have the following representation
  \begin{equation}\label{spectral-representation}
   u(x) = Eu(x) + \int_\X e^{ix\cdot\xi}Z(d\xi),
  \end{equation}
  where $Z$ is an orthogonal random measure (see~\cite[Chapter 1]{MR1009786} for specific results as well as \cite{MR0173945, MR0247660}, or \cite{MR2840012} for general theory of random measures).
  There also exists a finite measure $\sigma$ such that
   \begin{equation*}
    E(u(x))^2-(Eu(x))^2 = E\Big(\int_\X e^{ix\cdot\xi}Z(d\xi)\Big)^2 = \int_\X \sigma(d\xi) = \sigma(\X).
   \end{equation*}
   In this case we denote $\sigma = \Spec(u)$.
   One may also observe that $\mathcal{F}(\sigma)= B$.
   
   As in the general case~\eqref{bielecki}, for every $K\geq0$ 
  we define the Bielecki norm and the corresponding space
  \begin{align*}
   &\K{u}{p} = \sup_{t\geq0}e^{-tK}\|u(t)\|_p,
   &\!\mathcal{B}_{K,p} = \big\{u\in C(\T,\Iso_p) : \K{u}{p} < \infty\big\}.
  \end{align*}
     We use these to introduce the spaces of \emph{jointly isometry-invariant random fields}.
  \begin{definition}\label{jiso-def}
  Given an isometry $\phi\in\Phi$ and $u\in C(\T,\Iso_p)$ let us set $(\phi u)(t,x) = u(t,\phi(x))$.
  For $K\geq0$ and $p\geq1$ we define
  \begin{equation*}
      \JIso_{K,p} = \big\{u\in \mathcal{B}_{K,p} : u \= \phi u\ \text{for every $\phi\in\Phi$}\big\}.
  \end{equation*}
  \end{definition}
  Similarly to the case of $\Iso_p$, the space $\JIso_{K,p}$ is not linear, but it is complete in the norm we defined. In order to compensate for the lack of linearity we consider the following relation.
  \begin{definition}
    Let $u,v\in \JIso_{K,p}$ or $u,v\in\Iso_p$.
    We say that $u\sim v$ if for every isometry $\phi\in\Phi$ we have $(u,v) \= (\phi u,\phi v)$.
  \end{definition}
  It immediately implies that when $u,v\in\JIso_{K,p}$ such that $u\sim v$,
  we have $\alpha u + \beta v\in\JIso_{K,p}$ for every $\alpha,\beta\in\R$.
  The same observation holds in the space $\Iso_p$. Keep in mind that this
  relation is not transitive.
  \begin{remark}
   Let $u,v\in\JIso_{K,p}$ or $u,v\in\Iso_p$ and $u\sim v$, and take two measurable functions $f$ and $g$.
   Notice that because $(u,v)\= (\phi u, \phi v)$, we also have (see~\cite[Theorem 25.7]{MR1324786})
   \begin{equation*}
    (f(u,v),g(u,v))\= (f(\phi u,\phi v),g(\phi u,\phi v))
   \end{equation*}
   and therefore $f(u,v)\sim g(u,v)$.
   By similar arguments we have $f(u)\sim f(v)$ and $f(u)\sim g(u)$.
  \end{remark}
  \begin{remark}
  Let $u\in\JIso_{K,p}\cap C^1(\T,C_b(\X,\Lp))$ and $t,h\geq0$. Notice that $u(t)\sim u(t+h)$, hence $u(t+h)-u(t)\in\Iso_p$. 
    We may consider 
    \begin{equation}\label{derivative}
      \dt u (t) = \lim_{h\to 0} \frac{u(t+h)-u(t)}{h},
    \end{equation}
    where the limit is taken in the sense of the norm $\|\,\cdot\,\|_p$, and we have
    $\phi (\dt u) \= \dt u$ for every isometry $\phi\in\Phi$.
  \end{remark}
  \begin{definition}
   Let $u\in C_b(\X,\Lp)$. For a given $z\in\X$ we define the (normalised) directional derivative $\grad u$ as the random field such that
   \begin{equation}\label{derivative-def}
    \lim_{h\to0}\Big\|\frac{u(x+hz)-u(x)}{h|z|}-\grad u(x)\Big\|_p = 0, 
   \end{equation}
   whenever the limit exists.
   We say that $u\in C^1_b(\X,\Lp)$ if for every $z\in\X$ we have $\nabla_z u\in C_b(\X,\Lp)$.
  \end{definition}
  \begin{remark}\label{spectral-derivative}
Suppose that $u\in\Iso_2$ has the repesentation
$u(x)=Eu(x)+\int_\X e^{ix\cdot\xi}Z(d\xi)$ and $\grad u$ exists. It follows directly from identity~\eqref{derivative-def} that
 we have
\begin{equation*}
  \grad u(x) = \int_\X i\tfrac{z}{|z|}\cdot\xi\,e^{ix\cdot\xi}\,Z(d\xi).
\end{equation*}
\end{remark}
  \begin{lemma}\label{derivative-zero}
Let $u,v\in\Iso_2$ and $u\sim v$. If $u\in C^1_b(\X,\Ld)$ then for every $z\in\X$ we have
\begin{equation*}
E\grad u(x)v(x) = 0. 
\end{equation*}
\end{lemma}
\begin{proof}
Since we assume $u\sim v$, hence for every $x,y\in\X$ we have $Eu(y)v(0) = Eu(-y)v(0)$
and $Eu(x+y)v(x)=Eu(y)v(0)$.
Therefore
 \begin{multline*}
   E\grad u(x)v(x)
   = \lim_{h\to0} E \frac{u(x+hz)-u(x)}{h|z|}v(x)
   = \lim_{h\to0} E \frac{u(hz)-u(0)}{h|z|}v(0) \\
   = \lim_{h\to0} E \frac{u(-hz)-u(0)}{h|z|}v(0) = E\nabla_{-z}
    u(x)v(x).
 \end{multline*}
 On the other hand we have $\grad u  =-\nabla_{-z}u$, hence $E\grad u(x)v(x)=0$.
\end{proof}

\end{subsection}
  \begin{subsection}{Spectral moments}
   We introduce Sobolev-type spaces of isotropic random fields.
 \begin{definition}
  For every $\alpha\geq0$ we define the space
  \begin{equation*}
   \Iso^\alpha_2  = 
   \Big\{ u_0 \in \Iso_2 : \int_\X|\xi|^{2\alpha}d\sigma(\xi) < \infty,\ \text{where $\sigma=\Spec(u_0)$}\Big\}
  \end{equation*}
  supplemented with the norm $\|u_0\|_{\alpha,2}^2 = \int_\X \big(1+|\xi|^{2\alpha}\big)d\sigma(\xi)$. 
 \end{definition}
 \begin{proposition}\label{iso-c1}We have the following embeddings
 \begin{enumerate}
  \item $\Iso_2^1\subset C^1_b(\X,\Ld)$
  \item $\Iso^\alpha_2\subset\Iso^\beta_2$ for every $\alpha\geq\beta$. 
 \end{enumerate}
 \end{proposition}
 \begin{proof}
  Let $\sigma$ be the spectral measure and $Z$ be the orthogonal random measure corresponding to an arbitrary
  $u_0\in\Iso_2$.
  
  Suppose $u_0\in\Iso^1_2$ and take an arbitrary $z\in\X$.
  Then it follows from the Cauchy-Schwarz inequality and Remark~\ref{spectral-derivative} that
  \begin{multline*}
   E(\grad u_0(x))^2 = E\,\Big(\int_\X i\tfrac{z}{|z|}\cdot\xi\,e^{ix\cdot\xi}\,Z(d\xi)\Big)^2 \\
   = \int_\X \big|i\tfrac{z}{|z|}\cdot\xi\big|^2 d\sigma(\xi)
   \leq \int_\X|\xi|^{2}d\sigma(\xi) \leq\|u_0\|_{1,2}^2.
  \end{multline*} The continuity follows from a similar calculation and the Lebesgue dominated convergence theorem.
  This concludes the proof of the first inclusion.
  
  To prove the second claim it suffices to notice that the measure $\sigma$ is finite and for $\alpha \geq \beta$ the function $\psi(|x|) = |x|^{\beta/\alpha}$ is concave. Therefore by the Jensen inequality
  we have
  \begin{equation*}
   \frac{1}{\sigma(\X)}\int_\X\psi(|\xi|^{2\alpha})d\sigma(\xi)\leq\psi\Big(\frac{1}{\sigma(\X)}\int_\X|\xi|^{2\alpha}d\sigma(\xi)\Big),
  \end{equation*}
  hence $\|u_0\|_{\beta,2}\leq C\|u_0\|_{\alpha,2}$ for some constant $C>0$.
 \end{proof}
    \begin{lemma}\label{derivative-estimate}
     Let $u\in\Iso_2$. We have $u\in\Iso^1_2$ if and only if there exist $c>0$ and $\epsilon>0$ such that 
      \begin{equation}\label{derivative-estimate-eq}
        E\Big|\frac{u(x+hz)-u(x)}{h|z|}\Big|^2 \leq c
      \end{equation}
      for every $h\in(0,\epsilon)$ and every $z\in\X$.
    \end{lemma}
    \begin{proof}Let $\sigma=\Spec(u)$. Since we assume $u\in\Iso_2$, we have
     \begin{multline}\label{spectral-something}
        E\Big|\frac{u(x+hz)-u(x)}{h|z|}\Big|^2 = E\Big|\int_\X\frac{e^{i(x+hz)\cdot\xi}-e^{ix\cdot\xi}}{h|z|}Z(d\xi)\Big|^2 \\
        = \int_\X\Big|\frac{e^{i(x+hz)\cdot\xi}-e^{ix\cdot\xi}}{h|z|}\Big|^2d\sigma(\xi).
      \end{multline}
      Suppose that \eqref{derivative-estimate-eq} holds for some $c>0$, every $h\in(0,\epsilon)$ and every $z\in\X$. The Fatou lemma gives us
      \begin{equation*}
       \int_\X |\xi|^2 d\sigma(\xi) \leq \liminf_{h\to0}\int_\X\Big|\frac{e^{i(x+hz)\cdot\xi}-e^{ix\cdot\xi}}{h|z|}\Big|^2d\sigma(\xi),
      \end{equation*}
      which implies
      \begin{equation*}
       \int_\X |\xi|^2 d\sigma(\xi) \leq c,
      \end{equation*}
     hence $u\in\Iso^1_2$.
      Now suppose that $u\in\Iso^1_2$. By definition we have $\int_\X|\xi|^{2}d\sigma(\xi) = \|u\|_{1,2}^2- \|u\|_{2}^2$.
      Because of a simple inequality $|e^{ix\cdot y}-1|\leq|x\cdot y|\leq|x||y|$ for every $x,y\in\X$, we have
      \begin{equation*}
       \Big|\frac{e^{i(x+hz)\cdot\xi}-e^{ix\cdot\xi}}{h|z|}\Big|^2\leq|\xi|^2
      \end{equation*}
      and therefore
      \begin{equation*}
       \int_\X\Big|\frac{e^{i(x+hz)\cdot\xi}-e^{ix\cdot\xi}}{h|z|}\Big|^2d\sigma(\xi)
       \leq \|u\|_{1,2}^2- \|u\|_{2}^2.
      \end{equation*}
      Combining this with \eqref{spectral-something} we get the estimate \eqref{derivative-estimate-eq}.
    \end{proof}
    \begin{corollary}\label{lipschitz}
     If $u\in\Iso^1_2$ and $f:\R\to\R$ is Lipschitz then $f(u)\in\Iso^1_2$.
    \end{corollary}
    \begin{proof}
    Because $f$ is Lipschitz we have
    \begin{equation*}
      \Big\|\frac{f(u(x+hz))-f(u(x))}{h|z|}\Big\|_2 \leq L \Big\|\frac{u(x+hz)-u(x)}{h|z|}\Big\|_2.
    \end{equation*}
    We conclude by applying Lemma~\ref{derivative-estimate} twice.
    \end{proof}
  \end{subsection}
    \end{section}
    \begin{section}{Linear equation}\label{linear-equation}
      In this section we discuss properties of solutions to the linear Cauchy problem
      \begin{equation}\label{heat}
        \left\{
        \begin{aligned}
          &\dt u + \fLap u = 0\quad&\text{on $\T\times\X$}, \\
          &u(0) \= u_0\quad&\text{on $\X$}.
        \end{aligned}
        \right.
      \end{equation}
      Before we study the solutions themselves we need to describe in detail some of the objects we work with.
      \begin{definition}\label{flap}
      Let $u\in C_c^\infty(\X,\R)$ and $s\in(0,1]$. We define the fractional Laplace operator
     \begin{equation*}
      \fLap u(x) = \mathcal{F}\{|\xi|^{2s}\mathcal{F}^{-1}u(\xi)\}.
     \end{equation*}
     \end{definition}
     \begin{definition}\label{iso-fractional-laplacian}
     For $u_0\in\Iso^{2s}_2$, such that $u_0=Eu_0(x)+ \int_\X e^{ix\cdot\xi}Z(d\xi)$, we define
         \begin{equation*}
             -\fLap u_0(x) = \int_\X |\xi|^{2s}e^{ix\cdot\xi}Z(d\xi).
         \end{equation*}
     \end{definition}
     \begin{remark}\label{flap-prop}
      Notice that we have
         \begin{equation*}
             E\fLap u_0(x) = 0
         \end{equation*}
         and $\Spec(\fLap u_0)=|\xi|^{4s}\sigma(d\xi)$.
     \end{remark}
     By $\{P_t\}_{t\geq0}$ we denote the usual semigroup of linear operators
     generated by the fractional Laplacian (regardless of the chosen parameter $s$, which we assume to be fixed).
         We define the action of the operators $P_t$ on the space of isometry-invariant random fields
         in the following fashion.
         \begin{definition}\label{semigroup-definition}
         For $u_0\in\Iso_2$, such that $u_0(x)=Eu_0(x)+\int_\X e^{ix\cdot\xi}Z(d\xi)$, we define
         \begin{equation*}
            P_tu_0(x) = Eu_0(x)+\int_\X e^{ix\cdot\xi-t|\xi|^{2s}}Z(d\xi).
         \end{equation*}
         \end{definition}
         \begin{remark}
         Notice that 
         \begin{equation*}
             \|P_tu_0\|_2^2 = (Eu_0(x))^2+\int_\X e^{-2t|\xi|^{2s}}\,\sigma(d\xi) \leq \|u_0\|_2^2,
         \end{equation*}
         which shows that $P_t$ is a contractive operator on $\Iso_2$. This expression also shows that
         $\Spec(P_tu_0)=e^{-2t|\xi|^{2s}}\sigma(d\xi)$. By comparing Definition~\ref{semigroup-definition} with representation~\eqref{spectral-representation} we see that the semigroup property is preserved as well
         \begin{multline*}
          P_r\big(P_tu_0(x)\big) = EP_tu_0(x)+\int_\X e^{ix\cdot\xi-r|\xi|^{2s}}\big(e^{-t|\xi|^{2s}}Z(d\xi)\big)\\
          = Eu_0(x)+\int_\X e^{ix\cdot\xi-(r+t)|\xi|^{2s}}Z(d\xi) = P_{r+t}u_0(x).
         \end{multline*}
         \end{remark}
         \begin{proposition}\label{semigroup-limit}
          If $u_0\in\Iso^s_2$ then 
          \begin{equation*}
           \lim_{h\to0} E\frac{u_0-P_hu_0}{h}u_0 = E\big((-\Delta)^{s/2}u_0\big)^2.
          \end{equation*}
         \end{proposition}
         \begin{proof}
           Let $\sigma$ be the spectral measure and $Z$ be the orthogonal random measure corresponding to $u_0$. We have
	   \begin{multline*}
             E\frac{u_0-P_hu_0}{h}u_0 
	     =\frac{1}{h}E\int_\X\big(e^{ix\cdot\xi}-e^{ix\cdot\xi-h|\xi|^{2s}}\big)Z(d\xi)\int_\X e^{ix\cdot\xi}Z(d\xi)\\
             = \frac{1}{h}\int_\X \big(1-e^{-h|\xi|^{2s}}\big)\sigma(d\xi).
           \end{multline*}
           Because we assume $u_0\in\Iso^s_2$ we may use the Lebesgue dominated convergence theorem and pass to the limit to obtain
           \begin{equation*}
           \lim_{h\to0}\frac{1}{h}\int_\X \big(1-e^{-h|\xi|^{2s}}\big)\sigma(d\xi) = \int_\X |\xi|^{2s}\sigma(d\xi).
           \end{equation*}
           On the other hand,
           \begin{equation*}
             \int_\X |\xi|^{2s}\sigma(d\xi) 
             = E\Big(\int_\X |\xi|^s e^{ix\cdot\xi}Z(d\xi)\Big)^2 = E\big((-\Delta)^{s/2}u_0\big)^2
           \end{equation*}
           (cf. Remark~\ref{flap-prop}).
         \end{proof}

         In addition to such ``spectral'' framework
         we may employ a direct approach as well.
	To this end we introduce the kernel $p_t$ of the operator $P_t$ defined by the formula
	\begin{equation*}
	 p_t(x) = \int_\X e^{ix\cdot\xi-t|\xi|^{2s}}\,d\xi.
	\end{equation*}
	It is well-known that for $s\in(0,1]$ and $t\geq0$ the function $p_t$ is positive, radially symmetric and the function $t\mapsto p_t(y)$ is continuous. For every $t>0$ we also have $\int_\X p_t(x)\,dx=1$.

	The following lemma provides a connection between the two approaches.
       \begin{lemma}\label{kernel-linfty}
          Let $u_0\in\Iso_p$ for some $2\leq p\leq\infty$. Then for every $t>0$ we have
          \begin{equation*}
           P_tu_0(x) = \int_\X p_t(y)u_0(x-y)\,dy,
          \end{equation*}
          where the integral is understood in the Bochner sense on functions in the space $C_b(\X,\Lp)$. Moreover, $\|P_tu_0\|_p\leq \|u_0\|_p$ for all $t\geq0$.
      \end{lemma}
  \begin{proof}
    Let $u_0(x) = Eu_0(x)+\int_\X e^{ix\cdot\xi}\,Z(d\xi)$. Following the general definition we have
    \begin{multline*}
      P_tu_0(x) = Eu_0(x)+\int_\X e^{ix\cdot\xi-t|\xi|^{2s}}Z(d\xi) \\
      = Eu_0(x)+\int_\X\int_\X p_t(y) e^{-iy\xi}\,dy\, e^{ix\cdot\xi}Z(d\xi).
    \end{multline*}
    By the Fubini theorem and the fact that $\int_\X p_t(y)\,dy = 1$ we then obtain
    \begin{multline*}
      P_tu_0(x) = Eu_0(x)+\int_\X p_t(y) \int_\X e^{-iy\xi}e^{ix\cdot\xi}Z(d\xi)\,dy \\
      = \int_\X p_t(y)\Bigg(Eu_0(x)+\int_\X e^{i(x-y)\cdot\xi}Z(d\xi)\Bigg)\,dy
      = \int_\X p_t(y)u_0(x-y)\,dy.
    \end{multline*}
    The last integral is convergent in the Bochner sense because we have
    \begin{equation}\label{norm-estimate}
     \int_\X \|p_t(y)u_0(x-y)\|_{\Lp}\,dy = \|u_0\|_p\int_\X p_t(y)\,dy = \|u_0\|_p,
    \end{equation}
    which also confirms that $\|P_tu_0\|_p\leq \|u_0\|_p$.
  \end{proof}
  \begin{proposition}\label{iso-infinity}
  If $u_0\in\Iso_\infty$ then $P_t u_0(x,\omega) = \int_\X p_t(y)u_0(x-y,\omega)\,dy$ for every $t>0$ and almost every $\omega\in\Omega$.
  \end{proposition}
  \begin{proof}
  Notice that the operator $T_\omega:C(\X,L^\infty(\Omega))\to L^\infty(\X,\Bor(\X),dx)$ defined as
  $T_\omega u = u(\cdot\,,\omega)$ is bounded. By the Hille theorem we then have
   \begin{multline*}
     P_t u_0(x,\omega) = T_\omega\Big(\int_\X p_t(y)u_0(x-y)\,dy\Big) \\= 
     \int_\X p_t(y)T_\omega \big(u_0(x-y)\big)\,dy = \int_\X p_t(y)u_0(x-y,\omega)\,dy.
   \end{multline*}

  \end{proof}
 \begin{remark}\label{l-infinity}
  Notice that without additional assumptions, an individual realisation $u_0(x,\omega)$
  of the random field $u_0$ may not be integrable, such that for a given $\omega\in\Omega$, the Lebesgue integral
  $\int_\X p_t(y)u_0(x-y,\omega)\,dy$ may not exist. This is the main reason
  why we cannot consider solutions for \emph{every single realisation} of the initial data separately
  and then \emph{average them out} to get our results.
  
 \end{remark}

  \begin{lemma}\label{jiso}
      For every $2\leq p \leq \infty$ and every $K\geq0$ if $u_0\in\Iso_p$ then $P_tu_0\in\JIso_{K,p}$.
  \end{lemma}
  \begin{proof}
      By Lemma~\ref{kernel-linfty} we have
      \begin{equation*}
       P_tu_0(x) = \int_\X p_t(y)u_0(x-y)\,dy.
      \end{equation*}
      Consider a sequence $t_n\to t$. By the continuity of the function $t\mapsto p_t(y)$ and the Lebesgue dominated convergence theorem we obtain
      \begin{multline*}
       \lim_{n\to\infty}\|P_tu_0(x) - P_{t_n}u_0(x)\|_p \leq
       \lim_{n\to\infty}\int_\X\big\|(p_t(y)-p_{t_n}(y))u_0(x-y)\big\|_p\\
       =\|u_0\|_p\lim_{n\to\infty}\int_\X\big|(p_t(y)-p_{t_n}(y))\big|  =0.
      \end{multline*}
      For every $K\geq0$ and $2\leq p\leq\infty$ identity \eqref{norm-estimate} gives us the estimate
      \begin{equation*}
       \sup_{t\geq0}e^{-tK}\|P_tu_0(x)\|_p \leq \sup_{t\geq0}e^{-tK}\int_\X \|p_t(y)u_0(x-y)\|_p\,dy  \leq  \|u_0\|_p,
      \end{equation*}
      which shows that $P_tu_0\in\mathcal{B}_{K,p}$.
      
      Let $\phi\in\Phi$ be an isometry and $Z$ be the orthogonal random measure corresponding to $u_0$. Then
      because $u_0\in\Iso_2$ we have
      \begin{multline*}
        \phi(P_tu_0)(x) = \phi\Big(\int_\X e^{ix\cdot\xi-t|\xi|^{2s}}\,Z(d\xi)\Big) \\=
        \int_\X e^{i\phi(x)\cdot\xi-t|\xi|^{2s}}\,Z(d\xi)  = P_t u_0(\phi(x)) \= P_t u_0(x),
      \end{multline*}
      which confirms that $P_tu_0\in\JIso_{K,p}$ (see Definition \ref{jiso-def}).
  \end{proof}
  Let us now show a regularising effect of the linear semigroup.
   \begin{lemma}\label{linear-regularity}
       If $u_0\in\Iso_2$ then $P_tu_0\in C((0,\infty),\Iso^\alpha_2)$ for every $\alpha\geq 0$. Moreover
       there exists a constant $c_{s,\alpha}$ such that for every $t>0$
       \begin{equation*}
        \|P_tu_0\|_{\alpha,2}\leq (1+c_{s,\alpha}t^{-\alpha/2s})\|u_0\|_{2}.
       \end{equation*}
      \end{lemma}
      \begin{proof}
        Let $\sigma = \Spec(u_0)$. Keep in mind it is a finite measure and $\sigma(\X) = \|u_0\|_2^2$.
        We then have $\Spec(P_tu_0) =  e^{-2t|\xi|^{2s}}\sigma(d\xi)$ and
        \begin{equation*}
          \|P_tu_0\|_{\alpha,2}^2 = \int_\X (1+|\xi|^{2\alpha})\, e^{-2t|\xi|^{2s}} \sigma(d\xi) 
          \leq \big(1+\sup_{\xi\in\X}|\xi|^{2\alpha} e^{-2t|\xi|^{2s}}\big)\sigma(\X).
        \end{equation*}
         Notice that for every $t>0$ and $\alpha\geq0$ we have
        \begin{equation}\label{cs}
          \sup_{\xi\in\X}|\xi|^{2\alpha} e^{-2t|\xi|^{2s}} = (2t)^{-\alpha/s}\sup_{\xi\in\X}|\xi|^{2\alpha} e^{-|\xi|^{2s}} = c_{s,\alpha}^2\,t^{-\alpha/s},
        \end{equation}
        which shows that
        \begin{equation*}
         \|P_tu_0\|_{\alpha,2}
         \leq \|u_0\|_{2}\sqrt{1+c_{s,\alpha}^2\,t^{-\alpha/s}}
         \leq \|u_0\|_{2}\big(1+c_{s,\alpha}\,t^{-\alpha/{2s}}\big).
        \end{equation*}
        Finally, by the Lebesgue dominated convergence theorem, for every $t>0$ we obtain
         \begin{equation*}
      \lim_{\tau\to t}\|P_tu_0- P_\tau u_0\|_{\alpha,2}^2 
      = \lim_{\tau\to t} \int_\X(1+ |\xi|^{2\alpha}) (e^{-2t|\xi|^{2s}}-e^{-2\tau|\xi|^{2s}})\,d\sigma(\xi) = 0,
     \end{equation*}
     which confirms the continuity.
      \end{proof}
  \begin{lemma}\label{linear-regularity1}
       If $u_0\in\Iso_2^\alpha$ 
       then $\grad P_tu_0\in C((0,\infty),\Iso^\alpha_2)$ for every $\alpha\geq 0$.
       Moreover
       there exists a constant $c_{s}$ (independent of $\alpha$) such that for every $t>0$
       \begin{equation*}
        \|\grad P_tu_0\|_{\alpha,2}\leq c_{s}t^{-1/2s}\|u_0\|_{\alpha,2}.
       \end{equation*}
  \end{lemma}
  \begin{proof}
        The proof is almost identical to that of Lemma~\ref{linear-regularity}.
	Let $\sigma = \Spec(u_0)$. 
        Then, because of the Cauchy-Schwarz inequality and identity~\eqref{cs}, we have 
        \begin{multline*}
          \|\grad P_tu_0\|_{\alpha,2}^2 = \int_\X (1+|\xi|^{2\alpha})\big|\tfrac{z}{|z|}\cdot\xi\big|^2\, e^{-2t|\xi|^{2s}} \sigma(d\xi) \\
          \leq \sup_{\xi\in\X}|\xi|^{2} e^{-2t|\xi|^{2s}}\int_\X(1+|\xi|^{2\alpha})\sigma(d\xi)= c_{s}^2\,t^{-1/s}\|u_0\|_{\alpha,2}^2.
        \end{multline*}
        The continuity follows in a similar fashion.
  \end{proof}
  \begin{lemma}\label{semigroup-estimate}
    There exists a constant ${c}_s$ (independent of $p$) such that
    if $u_0\in\Iso_p$ then
    \begin{equation*}
     \|\grad P_t u_0\|_p \leq {c}_s t^{-1/2s}\|u_0\|_p
    \end{equation*}
    for every $t>0$.
  \end{lemma}
  \begin{proof}
   Notice that
   \begin{equation*}
    \grad P_t u_0 = \int_\X\grad p_t(y)u_0(x-y)\,dy.
   \end{equation*}
   It is well-known that (see~\cite{MR2373320} or \cite{MR3211862})
   \begin{equation*}
    p_t(y) = t^{-d/2s}	p_1(t^{-1/2s}y)
   \end{equation*}
for every $t>0$ and 
   \begin{equation*}
    |\grad p_1(y)| \leq C(1+|y|)^{-(2s+d+1)}.
   \end{equation*}
   This allows us to estimate
   \begin{multline*}
    \|\grad P_t u_0\|_p
    \leq \int_\X \|\grad p_t(y)u_0(x-y)\|_p\,dy  
    = \|u_0\|_p\int_\X |\grad p_t(y)|\,dy \\
    \leq t^{-(d+1)/2s}\|u_0\|_p \,C\int_\X (1+|t^{-1/2s}y|)^{-(2s+d+1)}\,dy\\
    =t^{-1/2s}\|u_0\|_p \,C\int_\X (1+|y|)^{-(2s+d+1)}\,dy.
   \end{multline*}
   The last integral is convergent and does not depend on $p$.
  \end{proof}
      The following theorem justifies calling $P_tu_0$ a solution to problem~\eqref{heat}.
  \begin{theorem}\label{semigroup-c1}
    If $u_0\in\Iso_2$ and $u(t) =P_tu_0$ then $u\in C^1((0,\infty),\Iso_2)$,
         $\dt u\in\JIso_{K,2}$ for every $K\geq 0$ and
    \begin{equation*}
      \dt u+ \fLap u = 0\quad\text{for every $t>0$}.
    \end{equation*}
  \end{theorem}
  \begin{proof}
    It follows from Lemma~\ref{linear-regularity} that $\fLap u(t)\in\Iso_2$ for every $t>0$.
    Let $\sigma=\Spec(u_0)$. Then $\Spec(u(t)) =  e^{-2t|\xi|^{2s}}\sigma(d\xi)$ and due to identity~\eqref{derivative} and Definition~\ref{iso-fractional-laplacian} it suffices to consider
    \begin{equation*}
      E\Big|\frac{u(t+h)-u(t)}{h}+\fLap u(t)\Big|^2 
      = \int_\X\Big(\frac{e^{-h|\xi|^{2s}}-1}{h}-|\xi|^{2s}\Big)^2 e^{-2t|\xi|^{2s}}d\sigma(\xi).
    \end{equation*}
    Since $\sigma$ is a finite measure we may pass to the limit with $h\to0$ on both sides and use the
    Lebesgue dominated convergence theorem to obtain
    \begin{equation*}
     \dt u(t) = -\fLap u(t)
    \end{equation*}
    for every $t>0$. Moreover, we have (see estimate~\eqref{cs})
    \begin{equation*}
     \|\dt u(t)\|_2^2 = \int_\X|\xi|^{4s}e^{-2t|\xi|^{2s}}\sigma(d\xi)\leq c_s t^{-2}\|u_0\|_2^2,
    \end{equation*}
    therefore $\dt u \in \JIso_{K,2}$ for every $K\geq 0$.
  \end{proof}
  \begin{remark}
    The analysis of the linear problem is exposed in more detail in \cite{MR3628179}.
    However, here we use a different definition of solutions, which is better suited for the nonlinear case which we
    discuss in the sequel.
  \end{remark}
    \end{section}
    \begin{section}{Equation with Lipschitz nonlinearity}\label{sec-lipschitz}
     Let us consider the following initial value problem
     \begin{equation}\label{burgers}
      \left\{
      \begin{aligned}
        &\dt u + \fLap u = \grad f(u)\quad&\text{on $\T\times\X$}, \\
        &u(0) \= u_0\quad&\text{on $\X$}.
      \end{aligned}
      \right.
    \end{equation}
    Here we assume $s\in(\frac{1}{2},1]$, $u_0\in\Iso_p$ and the function $f:\R\to\R$ is to be Lipschitz, i.e. for every $x,y\in\R$ and some constant $L>0$ we have $|f(x)-f(y)| \leq L |x-y|$,
    and such that $f(0)=0$.
    In the following, by referring to~problem~\eqref{burgers}, we also quietly include these assumptions.
    \begin{subsection}{Existence of solutions}
  \begin{definition}\label{definition}
  Given $u\in\JIso_{K,2}$ for some $K\geq 0$, we define the following nonlinear operator
  \begin{equation*}
    F(u)(t) = P_{t}u(0) + \int_0^{t} \grad P_{t-\tau} f(u(\tau))\,d\tau.
  \end{equation*} 
   Let $u_0\in\Iso_p$. For $K\geq0$ we say that $u\in\JIso_{K,p}$ is a solution to problem~\eqref{burgers} if $F(u)=u$ and $u_0\=u(0)$.
  \end{definition}
  \begin{lemma}\label{Fimage}
   For every $K>0$ and $2\leq p\leq\infty$ if $u\in\JIso_{K,p}$ then $F(u)\in\JIso_{K,p}$.
  \end{lemma}
  \begin{proof}
   First we estimate the norm to check if $F(u)\in \mathcal{B}_{K,p}$.  We have
   \begin{equation*}
       \K{F(u)}{p}
       = \sup_{t\geq0}e^{-tK}\Big\|P_{t}u(0) + \int_0^{t} \grad P_{t-\tau} f(u(\tau))\,d\tau\Big\|_p.
    \end{equation*}
    It follows from Lemma~\ref{semigroup-estimate} that 
    \begin{multline}\label{norm-est1}
            \Big\|\int_0^{t} \grad P_{t-\tau} f(u(\tau))\,d\tau\Big\|_p
            \leq c_sL\int_0^{t} (t-\tau)^{-1/2s}\|u(\tau)\|_p\,d\tau\\
            \leq c_sL\,\Big(\sup_{0<\tau<t}e^{-\tau K}\|u(\tau)\|_p\Big) \int_0^{t} (t-\tau)^{-1/2s} e^{\tau K}\,d\tau.
    \end{multline}
    Using the $\Gamma$ function we estimate
    \begin{multline}\label{gamma}
     \int_0^{t} (t-\tau)^{-1/2s} e^{-K(t-\tau)}d\tau < K^{-1+\frac{1}{2s}}\int_0^{\infty} z^{-1/2s}e^{-z}\,dz \\= K^{-1+\frac{1}{2s}}\Gamma(1-\tfrac{1}{2s}),
    \end{multline}
    which gives us
    \begin{align*}
     \K{F(u)}{p} \leq \K{P_tu_0}{p} + c_sL K^{-1+\frac{1}{2s}}\Gamma(1-\tfrac{1}{2s}) \K{u}{p}.
    \end{align*}
   Then for every isometry $\phi\in\Phi$ we observe
   \begin{multline*}
    \phi F(u) = \phi\Big(P_t u(0) + \int_0^t\grad P_{t-\tau}f(u(\tau))\,d\tau\Big) \\
    =\phi\big(P_t u(0)\big) + \phi\Big(\int_0^t\grad P_{t-\tau}f(u(\tau))\,d\tau\Big) \\
    = P_t (\phi u(0)) + \int_0^t\grad P_{t-\tau}f(\phi u(\tau))\,d\tau = F(\phi u).
   \end{multline*}
   Finally, since $\phi u \= u$, we have $F(\phi(u)) \= F(u)$ and therefore $\phi F(u) \= F(u)$.
  \end{proof}
  \begin{lemma}\label{contraction}
    If $u,v\in\JIso_{K,p}$ and $u\sim v$ then
    \begin{equation*}
        \K{F(u)-F(v)}{p} \leq \|u(0)-v(0)\|_p + c_sLK^{-1+\frac{1}{2s}}\Gamma(1-\tfrac{1}{2s})\K{u-v}{p}.
    \end{equation*}
  \begin{proof}
    First we notice that $f(u)\sim f(v)$ and therefore
    \begin{equation*}
     \grad P_t f(u) - \grad P_tf(v) = \grad P_t (f(u)-f(v)).
    \end{equation*}
    This gives us the following inequality
     \begin{multline*}
      \|F(u)(t)-F(v)(t)\|_p 
      \\\leq \|P_t(u(0)-v(0))\|_p +\int_0^{t}\big\|\grad P_{t-\tau} (f(u(\tau))-f(v(\tau)))\big\|_p\,d\tau.
      \end{multline*}
      Similarly to estimate \eqref{norm-est1} we have
      \begin{multline*}
      \int_0^{t}\big\|\grad P_{t-\tau} (f(u(\tau))-f(v(\tau)))\big\|_p\,d\tau\\
      \leq c_sL\,\Big(\sup_{0\leq\tau\leq t}e^{-\tau K}\|u(\tau)-v(\tau)\|_p\Big)\int_0^{t} (t-\tau)^{-1/2s}e^{\tau K}\,d\tau.
    \end{multline*}
    We combine it with estimate \eqref{gamma} and Lemma~\ref{kernel-linfty} to obtain
    \begin{multline*}
      \K{F(u)-F(v)}{p} = \sup_{t\geq0}e^{-tK}\|F(u)(t)-F(v)(t)\|_p \\
      \leq \|u(0)-v(0)\|_p + c_sLK^{-1+\frac{1}{2s}}\Gamma(1-\tfrac{1}{2s})\K{u-v}{p}.\qedhere
    \end{multline*}
  \end{proof}
  \end{lemma}
  \begin{theorem}\label{existence}
    Let $\frac{1}{2}<s\leq 1$ and $2\leq p\leq\infty$. There exists a constant $K_0$ such that
    for every $u_0\in\Iso_p$ and every $K\geq K_0$ the sequence 
     \begin{equation}\label{picard}
      \left\{
      \begin{aligned}
        &u_1 = P_t u_0, \\
        &u_{n+1}=F(u_{n}) = F^{n}(u_1)
      \end{aligned}
      \right.
    \end{equation}
    converges in $\JIso_{K,p}$ to a solution of problem~\eqref{burgers}.
  \end{theorem}
  \begin{proof}
  It follows from Lemma~\ref{jiso} that $u_1\in\JIso_{K,p}$ for every $K>0$,
  and then from Lemma~\ref{Fimage} that $\{u_n\}\subset\JIso_{K,p}$.
  Let us notice that $u_n\sim u_m$ and $u_n(0)=u_m(0)=u_0$.
  Thus by Lemma~\ref{contraction} we have
  \begin{equation*}
      \K{F(u_n)-F(u_m)}{p}\leq c_sLK^{-1+\frac{1}{2s}}\Gamma(1-\tfrac{1}{2s})\K{u_n-u_m}{p}.
  \end{equation*}
  We now choose such $K_0$ that $c_sLK_0^{-1+\frac{1}{2s}}\Gamma(1-\frac{1}{2s})<1$, depending on $s$ and $L$. 
  It follows from the Banach fixed point theorem that $\{u_n\}$ is a Cauchy sequence in the space
  $\JIso_{K,p}$ for every $K\geq K_0$ (we use the assumption $s>\frac{1}{2}$)
  and converges to some $u\in\JIso_{K,p}$ which is a fixed point of $F$.
  \end{proof}
\begin{definition}
 We refer to the solution constructed in Theorem~\ref{existence} as the Picard solution.
\end{definition}
  \begin{corollary}\label{uniqueness}
      If $u$ is the Picard solution to~\eqref{burgers} then for every $h>0$ and every $t>0$
      we have
      \begin{equation*}
          u(t+h) = P_hu(t)+\int_t^{t+h}\grad P_{t+h-\tau}f(u(\tau))\,d\tau.
      \end{equation*}
  \end{corollary}
  \begin{proof}
   Let $u_n$ be the sequence of Picard iterations as defined in \eqref{picard}.
   Since $u_{n} = F(u_{n-1})$ for $n\geq 2$ we have
   \begin{align*}
    P_hu_n(t) &= P_hP_tu_0 + P_h\int_0^t\grad P_{t-\tau}f(u_{n-1}(\tau))\,d\tau\\
        &= P_{t+h}u_0 + \int_0^t\grad P_{t+h-\tau}f(u_{n-1}(\tau))\,d\tau
   \end{align*}
    and
    \begin{equation*}
     u_n(t+h) = P_{t+h}u_0 + \int_0^{t+h}\grad P_{t+h-\tau}f(u_{n-1}(\tau))\,d\tau.
    \end{equation*}
    Hence
    \begin{equation*}ope
     u_n(t+h) - P_hu_n(t) = \int_t^{t+h}\grad P_{t+h-\tau}f(u_{n-1}(\tau))\,d\tau
    \end{equation*}
    and finally, because $\lim_{n\to\infty} \|u_n -u\|_2 =0$, we get
    \begin{equation*}
     u(t+h) - P_hu(t) = \int_t^{t+h}\grad P_{t+h-\tau}f(u(\tau))\,d\tau.
    \end{equation*}
  \end{proof}
  \begin{remark}\label{uniqueness-remark}
   In this paper we do not consider the question of uniqueness of solutions and
   indeed, Definition~\ref{definition} may be too relaxed to ascertain it.
   In~\cite{MR3628179} it is shown that the semigroup solution $P_tu_0$ is
   in fact the unique solution to the linear problem under additional
   continuity-in-time assumptions. In the remainder we only work with
   the Picard solutions to problem~\eqref{burgers}, which are well-defined.
  \end{remark}
    \end{subsection}
\end{section}
\begin{section}{Regularity of solutions}\label{apriori}
\begin{subsection}{Moment estimates}
     In the first part of this section we reproduce the second moment estimates presented in~\cite{MR0264252}.
     We are only able to do this while assuming higher regularity of the initial condition, namely $u_0\in\Iso_2^1$.
     Despite this limitation, the result is interesting because of an elegant identity described in Remark~\ref{covariance}. In the sequel we obtain weaker (but sufficient) estimates for all moments in a more general
     setting.
     
     \begin{lemma}\label{infinitesimal-linearisation}
     Suppose $u$ is the Picard solution to~\eqref{burgers}.
     Let $g:\R\to\R$ be a measurable function such that $g(u(t))\in\Iso_2$ for some $t\geq0$.
     Then for every $h>0$ we have 
     \begin{equation*}
      Eu(t+h)g(u(t)) = EP_hu(t)g(u(t)) + hR(h),
     \end{equation*}
     where $R:\T\to\R$ is a function such that $\lim_{h\to0}R(h)=0$.
    \end{lemma}
    \begin{proof}
     From Corollary~\ref{uniqueness} we have
     \begin{equation*}
      u(t+h) = P_hu(t) + \int_t^{t+h}\grad P_{t+h-\tau}f(u(\tau))\,d\tau.
     \end{equation*}
     We define $R(h) = h^{-1}E\int_t^{t+h}\grad P_{t+h-\tau}f(u(\tau))\,d\tau\, g(u(t))$.
     By the Lebesgue differentiation theorem and Lemma~\ref{derivative-zero} we obtain
     \begin{equation*}
      \lim_{h\to0}E\frac{1}{h}\int_t^{t+h}\grad P_{t+h-\tau}f(u(\tau))\,d\tau\, g(u(t)) = E \grad f(u(t)) g(u(t)) = 0.
     \end{equation*}
    \end{proof}
    \begin{corollary}\label{c1/2}
     Suppose $u_0\in\Iso_2$ and $u$ is the Picard solution to~\eqref{burgers}. Then
     \begin{equation*}
      \lim_{h\to0}E\frac{\big(u(t+h)-u(t)\big)^2}{h}=0.
     \end{equation*}
    \end{corollary}
    \begin{proof} Notice that if $a=b+c$ then $a^2=b^2 + c(b+a)$.
     It follows from Corollary~\ref{uniqueness} that
     \begin{equation*}
      u(t+h)^2 = \big(P_hu(t)\big)^2 + \Big(\int_t^{t+h}\grad P_{t+h-\tau}f(u(\tau))\,d\tau\Big)\big(P_hu(t)+u(t+h)\big).
     \end{equation*}
     Then by Lemma~\ref{infinitesimal-linearisation} we get
     \begin{align*}
      &E\big(u(t+h)-u(t)\big)^2 
      = Eu(t+h)^2-2Eu(t+h)u(t)+Eu(t)^2\\
      &= E\big(P_hu(t)\big)^2 + E\Big(\int_t^{t+h}\grad P_{t+h-\tau}f(u(\tau))\,d\tau\Big)\big(P_hu(t)+u(t+h)\big)
      \\&-2EP_hu(t)u(t)+Eu(t)^2+hR(h),
     \end{align*}
     where $\lim_{h\to0} R(h)=0$.
     This, combined with the Lebesgue differentiation theorem, implies
     \begin{equation*}
      \lim_{h\to0}E\frac{\big(u(t+h)-u(t)\big)^2}{h}=\lim_{h\to0}E\frac{\big(P_hu(t)-u(t)\big)^2}{h}=  0.
     \end{equation*}
    \end{proof}
    
    \begin{lemma}\label{regularity}
     If $u$ is the Picard solution to problem~\eqref{burgers} with $u_0\in\Iso^1_2$
     then $u\in C([0,\infty),\Iso^1_2)$.
    \end{lemma}
    \begin{proof}
     Let $u_n$ be the sequence of Picard iterations as defined in \eqref{picard}
     and let $K_0$ be such as in Theorem~\ref{existence}, i.e.
     \begin{equation}\label{k0}
      c_sLK_0^{-1+\frac{1}{2s}}\Gamma(1-\tfrac{1}{2s}) < 1,
     \end{equation}
     where $c_s$ is a known constant.
     
     Let us proceed by induction.
     We have
     \begin{equation*}
      \|u_1(t)\|_{1,2}^2 = \int(1+|\xi|^{2s})e^{-2t|\xi|^{2s}}\sigma(d\xi)3
      \leq \|u_0\|_{1,2}^2,
     \end{equation*}
     hence $u_1\in C([0,\infty),\Iso^1_2)$.

     Suppose $u_n\in C([0,\infty),\Iso^1_2)$ and
     $\sup_{t\geq0}e^{-tK_0}\|u_n(t)\|_{1,2}\leq n\|u_0\|_{1,2}$.
     By Lem\-ma~\ref{linear-regularity1} and Corollary~\ref{lipschitz} for every $t>0$ we have
     \begin{equation*}
      \|\grad P_{t} f(u_n(\tau))\|_{1,2} 
      \leq {c}_s t^{-1/2s}\|f(u_n(\tau))\|_{1,2} \leq {c}_sL t^{-1/2s}\|u_n(\tau)\|_{1,2}.
     \end{equation*}
     Therefore
     \begin{multline*}
      \|u_{n+1}(t)\|_{1,2} \leq \|u_1(t)\|_{1,2} + \int_0^{t} \|\grad P_{t-\tau} f(u_n(\tau))\|_{1,2}\,d\tau \\
      \leq \|u_1(t)\|_{1,2} + {c}_sL\int_0^{t} (t-\tau)^{-1/2s} \|u_n(\tau)\|_{1,2}\,d\tau.
     \end{multline*}
     Using estimates \eqref{norm-est1}, \eqref{gamma}, \eqref{k0} and the induction hypothesis, we thus obtain
     \begin{multline*}
      \sup_{t\geq0}e^{-tK_0}\|u_{n+1}(t)\|_{1,2}\\
      \leq \|u_1(t)\|_{1,2} + n{c}_sLK_0^{-1+\frac{1}{2s}}\Gamma(1-\tfrac{1}{2s})\|u_0\|_{1,2}\leq (n+1)\|u_0\|_{1,2}.
     \end{multline*}
     It follows that $u_{n+1}\in C([0,\infty),\Iso^1_2)$.

     Moreover, we obtain
     \begin{multline*}
      \sup_{t\geq0}e^{-tK_0}\|u_{n}(t)-u_m(t)\|_{1,2} \\\leq {c}_sLK^{-1+\frac{1}{2s}}\Gamma(1-\tfrac{1}{2s})\sup_{t\geq0}e^{-tK_0}\|u_{n-1}-u_{m-1}\|_{1,2}
     \end{multline*}
     and because of~\eqref{k0} it follows that $\{u_n\}$ is a Cauchy sequence in $C([0,T],\Iso^1_2)$ for every $T>0$ and $u\in C([0,\infty),\Iso^1_2)$.
    \end{proof}
    \begin{theorem}\label{rosenblatt}
        If $u$ is the Picard solution to~\eqref{burgers} with $u_0\in\Iso^1_2$ then
        \begin{equation*}
        \dt Eu(t)^2 = -2E\big((-\Delta)^{s/2}u(t)\big)^2.
    \end{equation*}
    \end{theorem}
    \begin{proof}
     From Lemma~\ref{infinitesimal-linearisation} we have
     \begin{equation}\label{nonlinear-semigroup-limit}
      E\frac{u(t+h)-u(t)}{h}u(t) = E\frac{P_hu(t)-u(t)}{h}u(t) + R(h),
     \end{equation}
     where $\lim_{h\to0}R(h)=0$.
     Notice that by Lemma~\ref{regularity} and Proposition~\ref{iso-c1} we have $u(t)\in\Iso^1_2\subset\Iso^s_2$ for $s\leq1$.
     Therefore on the right-hand side of equality \eqref{nonlinear-semigroup-limit} by Proposition~\ref{semigroup-limit} we get
     \begin{equation*}
      \lim_{h\to0}\Big(E\frac{P_hu(t)-u(t)}{h}u(t) + R(h)\Big) = -E\big((-\Delta)^{s/2}u(t)\big)^2.
     \end{equation*}
     This ensures the existence of the limit on the left-hand side
     and by Corollary~\ref{c1/2} we obtain
     \begin{align*}
      \lim_{h\to 0} &\frac{E(u(t+h)-u(t))}{h}u(t)
      = \frac{1}{2} \lim_{h\to 0} \frac{E\big(u(t+h)-u(t)\big)}{h}(u(t)+u(t+h))\\
      &= \frac{1}{2} \lim_{h\to 0} \frac{E\big(u(t+h)^2-u(t)^2\big)}{h}
      = \frac{1}{2}\dt Eu(t)^2.
     \end{align*}
    \end{proof}
    \begin{remark}\label{covariance}
    For the solution $u$ as in Theorem~\ref{rosenblatt} we define the covariance functional $B(t,y) = Eu(t,x)u(t,x+y)$. Let us notice that
     $\dt B(t,0) = \dt Eu(t,x)^2$ and according to Definition~\ref{flap} we have
     \begin{multline*}
      (-\Delta^s_y) B(t,y)|_{y=0} 
      = \int_\X |\zeta|^{2s}\int_\X e^{-iz\cdot\zeta}Eu(t,x)u(t,x+z)\,dz\,d\zeta\\
      = \int_\X |\zeta|^{2s}\int_\X e^{-iz\cdot\zeta}\int_\X e^{iz\cdot \xi}\,d\sigma_t(\xi)\,dz\,d\zeta
      = \int_\X |\zeta|^{2s}\sigma_t(d\zeta)\\
      = E\big((-\Delta)^{s/2}u(t)\big)^2. 
     \end{multline*}
     Therefore we may write the result of Theorem~\ref{rosenblatt} as 
     \begin{equation*}
      \dt B(t,0) = -2(-\Delta^s_y) B(t,y)|_{y=0}.
     \end{equation*}
     Note that this property of the functional $B(t,y)$ holds \emph{only} at $y=0$.
    \end{remark}
  \end{subsection}
  \begin{subsection}{Higher moments estimates}
  In order to prove the estimates similar to those obtained (for the second moment) in Theorem~\ref{rosenblatt} for moments $p>2$, we need to
  fall back to the classical theory applied pathwise because of issues with regularity. As indicated in Remark~\ref{l-infinity}, this is
  only possible in the space $\Iso_\infty$. Accordingly, we will work with cut-off initial data
  first and then reach the general case by approximation.  
    \begin{lemma}\label{classical-solutions}
     If $u_0\in\Iso_\infty$ and $f\in C^\infty(\R)$ then
     for every $\omega\in\Omega$  there exists a unique function 
     $u^\omega(t,x)\in C^\infty((0,\infty)\times\X)\cap L^\infty(\T\times\X,dx)$ which is a (classical) solution to
     \begin{equation}\label{classical-equation}
     \left\{
      \begin{aligned}
        &\dt u + \fLap u = \grad f(u)\quad&\text{on $\T\times\X$}, \\
        &u^\omega(0) = u_0(x,\omega)\quad&\text{a.e. on $\X$}.
      \end{aligned}
      \right.
      \end{equation}
    \end{lemma}
    \begin{proof}
     We know from Proposition~\ref{iso-infinity} that $u_0(x,\omega)\in L^\infty(\X,dx)$ for every 
     $\omega\in\Omega$. Because we assume $f\in C^\infty(\R)$, it follows from~\cite[Theorem 1.1]{MR2019032} that there exists a unique solution $u^\omega$ to problem~\eqref{classical-equation}, such that
     $u^\omega\in C^\infty((0,\infty)\times\X)\cap L^\infty(\T\times\X,dx)$.
    \end{proof}
    \begin{lemma}\label{classical-solutions-equiv}
     Let $f$ be Lipschitz. Suppose $u$ is the Picard solution to~\eqref{burgers} with $u_0\in\Iso_\infty$ and $f\in C^\infty(\R)$.
     Let $u^\omega$ be defined as in Lemma~\ref{classical-solutions} for every $\omega\in\Omega$.
     Then $u^\omega(t,x) = u(t,x,\omega)$ for almost all $\omega\in\Omega$.
    \end{lemma}
    \begin{proof}
    Recall we denote by 
    \begin{equation*}
    \mathcal{B}_{K,L^\infty(\X,dx)} = \{u\in C([0,\infty),L^\infty(\X,dx)):\sup_{t\geq0} e^{-tK}\|u(t)\|_{L^\infty(\X.dx)}\}
    \end{equation*}
    the Banach space of continuous functions endowed with an appropriate Bielecki norm.
    
    Let us proceed by induction. It follows from Proposition~\ref{iso-infinity} that
    \begin{equation*}
     u_1^\omega(t,x) = u_1(t,x,\omega) = P_t u_0(x,\omega)    
    \end{equation*}
    for almost every $\omega\in\Omega$, as well as that for every $\omega\in\Omega$
    and every $K>0$ we have $u_1^\omega\in \mathcal{B}_{K,L^\infty(dx)}$.
    Let $K_0$ be such as in Theorem~\ref{existence}, the sequence $u_n$ be defined as in~\eqref{picard} and $u_n^\omega$ be analogous sequences of (non-random) Picard iterations starting with $u_1^\omega$ for every $\omega\in\Omega$. 
    
    Suppose we have $u_n(t,x,\omega) = u_n^\omega(t,x)$ for almost every $\omega\in\Omega$ and $u_n^\omega(t,x)\in \mathcal{B}_{K,L^\infty(dx)}$ for every $\omega\in\Omega$
    and every $K\geq K_0$.
    Then because of the Hille theorem (as in the proof of Proposition~\ref{iso-infinity}) we get
    \begin{equation*}
    \Big(\int_0^t \grad P_{t-\tau}f(u_n(\tau,x))\,d\tau\Big) (\omega)
    = \int_0^t \grad P_{t-\tau}f(u_n^\omega(\tau,x))\,d\tau.
    \end{equation*}
    It follows that $u_{n+1}(t,x,\omega) = u_{n+1}^\omega(t,x)$ and moreover
    \begin{multline*}
     \|u_{n+1}^\omega(t)\|_{L^\infty(dx)}
     \leq \|u^\omega_1(0)\|_{L^\infty(dx)} 
     + \int_0^t \|\grad P_{t-\tau}f(u_n^\omega(\tau))\|_{L^\infty(dx)}\,d\tau
     \\\leq \|u^\omega_1(0)\|_{L^\infty(dx)} 
     + c_sL\int_0^{t} (t-\tau)^{-1/2s}\|u_n^\omega(\tau)\|_{L^\infty(dx)}\,d\tau.
    \end{multline*}
    This gives us an estimate
    \begin{multline*}
     \K{u_{n+1}^\omega}{L^\infty(\X,dx)}\\\leq 
     \K{u_{1}^\omega}{L^\infty(\X,dx)} +
     c_sLK^{-1+\frac{1}{2s}}\Gamma(1-\tfrac{1}{2s})\K{u_{n}^\omega}{L^\infty(\X,dx)},
    \end{multline*}
    which, because we have $K\geq K_0$, asserts that $\{u_{n}\}\subset\mathcal{B}_{K_0,L^\infty(dx)}$.
    
    Consequently
    $u^\omega(t,x) = u(t,x,\omega)$, for almost every $\omega\in\Omega$, since both solutions are defined as limits of
    Picard iterations.
    \end{proof}
    \begin{proposition}\label{liskevich-semenov}
     Let $w\in\Iso_\infty$, $\theta:\X\to\R$ be a measurable function such that $|\theta|=1$ and $a+b=2$. Then
     \begin{equation*}
      EP_h\theta |w|^a \theta |w|^b \leq abEP_h|w||w|
     \end{equation*}
    \end{proposition}
    \begin{proof}
     Let us notice that, because $\theta(x)|w(x)|^a\theta(x)|w(x)|^b = |w(x)|^2$ and
     $p_h(x-y)=p_h(y-x)$, we have
     \begin{multline*}
	E|w|^2 - EP_h\theta |w|^a\theta |w|^b 
	= E|w|^2 - \int_\X p_h(x-y)\theta(y)|w(y)|^a\,\theta(x)|w(x)|^b\,dy\\
	 =\frac{1}{2}\int_\X\!\!\!\! p_h(x-y)E\big(\theta(y)|w(y)|^a\!-\theta(x)|w(x)|^a\big)\!
					   \big(\theta(y)|w(y)|^b\!-\theta(x)|w(x)|^b\big)dy.
     \end{multline*}
     Similarly,
     \begin{multline*}
	EP_h|w||w|
	= \int_\X p_h(x-y)E|w(y)|\,|w(x)|\,dy\\
	= E|w|^2
	- \frac{1}{2}\int_\X p_h(x-y)E\big(|w(y)|-|w(x)|\big)^2\,dy.
     \end{multline*}
     We may now use the following inequality (see \cite[Lemma II.5.5]{MR1218884}
     \begin{equation*}
      (\theta_1x^a-\theta_2y^a)(\theta_1x^b-\theta_2y^b)\geq ab(x-y)^2,\ \text{when $a+b=2$ and $|\theta_1|=|\theta_2|=1$}
     \end{equation*}
     to obtain the desired result.
    \end{proof}
    \begin{lemma}\label{cut-off-reg}
    If $u$ is the Picard solution to~\eqref{burgers} with $u_0\in\Iso_\infty$ and $f\in C^\infty(\R)$ then
     \begin{equation*}
      \dt E|u|^p \leq -4\frac{p-1}{p}E\big|(-\Delta)^{s/2}|u|^{p}\big|^2
     \end{equation*}
      for every $p\geq2$.
    \end{lemma}
    \begin{proof}
     From Lemma~\ref{infinitesimal-linearisation} we have
     \begin{multline*}
      E\frac{u(t+h)-u(t)}{h}|u|^{p-1}(t)\sgn u(t) \\= E\frac{P_hu(t)-u(t)}{h}|u|^{p-1}(t)\sgn u(t) + R(h),
      \end{multline*}
      where $\lim_{h\to0}R(h) = 0$.
      Let $w = |u|^{\frac{p-2}{2}}u$. Then
      \begin{equation*}
       P_hu\big(|u|^{p-1}\sgn u\big) = P_h \big(|w|^{2/p}\sgn w\big)\big(|w|^\frac{2(p-1)}{p}\sgn w\big)
      \end{equation*}
      and it follows from Proposition~\ref{liskevich-semenov} that
      \begin{equation*}
       E(P_hu(t)-u(t))|u|^{p-1}(t)\sgn u(t) \leq 4\frac{p-1}{p^2}E(P_h |w| -|w|)|w|.
      \end{equation*}
      Because $u(t)\in\Iso^1_2\cap\Iso_\infty$ for $t>0$ then by Corollary~\ref{lipschitz} $w\in\Iso^1_2\subset\Iso^s_2$ and we have
      \begin{equation*}
        \lim_{h\to0}E\frac{(P_h |w|-|w|)|w|}{h} = -E\big|(-\Delta)^{s/2}|w|\big|^2.
      \end{equation*}
      On the other hand, it follows from Lemmas~\ref{classical-solutions} and~\ref{classical-solutions-equiv}
      that $u(t,x,\omega)=u^\omega(t,x)$ for almost every $\omega\in\Omega$
      and $u^\omega\in C^\infty((0,\infty)\times\X)$, therefore
      \begin{equation*}
      \lim_{h\to0}E\frac{u(t+h)-u(t)}{h}|u|^{p-1}(t)\sgn u(t) = \frac{1}{p}\dt E|u|^p,
      \end{equation*}
      which yields the result.
    \end{proof}
    \begin{theorem}\label{rosenblatt-lp}
     If $u$ is the Picard solution to~\eqref{burgers} with $u_0\in\Iso_p$ and $f\in C^\infty(\R)$ then
     \begin{equation*}
      E|u(t)|^p \leq E|u_0|^p
     \end{equation*}
      for every $2\leq p\leq\infty$ and every $t\geq0$.
    \end{theorem}
    \begin{proof}
     We consider the sequence $u_0^n = h_n(u_0)$, where $h_n$ are cut-off functions $h_n(x) = \min\{|x|,n\}\sgn(x)$ and the sequence $u_n$ of solutions to problems
      \begin{equation*}
      \left\{
      \begin{aligned}
        &\dt u + \fLap u = \grad f(u)\quad&\text{on $\T\times\X$}, \\
        &u(0) \= u_0^n\quad&\text{on $\X$}.
      \end{aligned}
      \right.
    \end{equation*}
     as constructed in Theorem~\ref{existence}. Then $u_0^n\in\Iso_\infty$ and it follows from Lemma~\ref{cut-off-reg}
     that 
     \begin{equation*}
      E|u_n(t)|^p \leq E|u_0^n|^p
     \end{equation*}
     for every $n\geq1$, $t\geq0$ and $p\geq2$.
     By Lemma~\ref{contraction} we know that
     \begin{equation*}
     \lim_{n\to\infty}\|u_n-u\|_p = 0
     \end{equation*}
     and the result follows.
    \end{proof}
  \end{subsection}
\end{section}
\begin{section}{Nonlinearity with Polynomial Growth}\label{general}
    We consider the following initial value problem
     \begin{equation}\label{polyburgers}
      \left\{
      \begin{aligned}
        &\dt u + \fLap u = \grad f(u)\quad&\text{on $\T\times\X$}, \\
        &u(0) \= u_0\quad&\text{on $\X$}.
      \end{aligned}
      \right.
    \end{equation}
    Here we assume $s\in(\frac{1}{2},1]$, $u_0\in\bigcap_{p\geq2}\Iso_p$ and $f\in C^\infty(\R)$, $f(0)=0$ and $|f(u)-f(v)|\leq C|u-v|\big(|u|^q+|v|^q\big)$
    for some constants $C>0$ and $q\geq 1$.
    By referring to~problem~\eqref{polyburgers} we also quietly include these assumptions.
    \begin{subsection}{Existence of solutions}
      \begin{definition}\label{polydefinition}
          Given $u\in\bigcap_{p\geq2}\JIso_{K,p}$ for some $K\geq 0$, we define the following nonlinear operator
  \begin{equation*}
    F(u)(t) = P_{t}u(0) + \int_0^{t} \grad P_{t-\tau} f(u(\tau))\,d\tau.
  \end{equation*} 
   Let $u_0\in\bigcap_{p\geq2}\Iso_p$. For $K\geq0$ we say that $u\in\bigcap_{p\geq1}\JIso_{K,p}$ is a solution to problem~\eqref{polyburgers} if $F(u)=u$ and $u_0\=u(0)$.
  \end{definition}
  In order to prove the existence of solutions we need to consider a sequence of approximations
  defined in the following way. Take $u_0^n = h_n(u_0)$, where $h_n$ are cut-off functions $h_n(x) = \min\{|x|,n\}\sgn(x)$ and define the sequence of Picard solutions $u^n$ to problems
  \begin{equation}\label{cutoff}
      \left\{
      \begin{aligned}
        &\dt u^n + \fLap u^n = \grad f(h_n(u^n))\quad&\text{on $\T\times\X$}, \\
        &u^n(0) \= u_0^n\quad&\text{on $\X$}.
      \end{aligned}
      \right.
    \end{equation}
    The function $f(h_n(x))$ is Lipschitz and thus, because of Theorem~\ref{existence}, such a sequence exists.
    We call it the sequence of \emph{approximative solutions} throughout this section.
    \begin{remark}
     For every $n$, the solution $u^n$ of problem~\eqref{cutoff} is by definition a fixed point of the operator
     \begin{equation}\label{noperator}
      F_n(u)(t) = P_{t}u(0) + \int_0^{t} \grad P_{t-\tau} f(h_n(u(t)))\,d\tau.
     \end{equation}
    \end{remark}
     In the following proposition we show that the approximative solutions $u^n$ are 
     not only solutions to problems~\eqref{cutoff} but also solutions
     to problem~\eqref{polyburgers}, with the cut-off initial conditions $u_0^n$, respectively.
     In other words, it turns out we do not need to cut-off the function $f$, once we have constructed
     solutions $u^n$ as limits of Picard iterations linked to problems~\eqref{cutoff}.
    \begin{proposition}
     Let $u^n$ be the sequence of approximative solutions.
     Then $u^n$ are fixed points of the operator $F$.
    \end{proposition}
    \begin{proof}
     It follows from Theorem~\ref{rosenblatt-lp} that $\|u^n(t)\|_\infty \leq \|u_0^n\|_\infty\leq n$ for every $n\geq1$ and $t\geq0$. Therefore $f(h_n(u^n(t)) = f(u^n(t))$.
     The Picard solution $u^n$ is a fixed point of the operator $F_n$ given by \eqref{noperator}
     and hence
     $u^n = F_n(u^n)=F(u^n)$.
    \end{proof}
    \begin{lemma}\label{l1}
        The sequence of approximative solutions is convergent in $\JIso_{K,p}$ for every $p\geq2$ and every $K>0$.
    \end{lemma}
    \begin{proof}
      It follows from Lemma~\ref{classical-solutions-equiv} that $u_n(\omega)$ are also classical solutions for almost
      every $\omega\in\Omega$, therefore we may write
      \begin{equation*}
            \dt(u_n-u_m) = -\fLap (u_n-u_m) + \grad(f(u_n)-f(u_m)).
        \end{equation*}
        Let $g(x)=|x|^{p-1}\sgn(x)$. Because
        \begin{equation*}
        f(u_n)-f(u_m) = \bar{f}(u_n,u_m)\sim \bar{g}(u_n,u_m) = g(u_n-u_m),
        \end{equation*}
        it follows from Lemma~\ref{derivative-zero} that
        \begin{equation*}
            E \grad (f(u_n)-f(u_m))|u_n-u_m|^{p-1}\sgn(u_n-u_m) =  0,
        \end{equation*}
     hence
       \begin{equation*}
            E\dt(u_n-u_m) g(u_n-u_m) = -E\fLap (u_n-u_m)g(u_n-u_m).
	\end{equation*}
	Notice that $g$ is a convex function, therefore
	\begin{equation*}
	 \fLap g(u_n(\omega)-u_m(\omega))\leq\fLap (u_n(\omega)-u_m(\omega))g(u_n(\omega)-u_m(\omega)) 
	\end{equation*}
	for almost every $\omega\in\Omega$ (see~\cite[Theorem 1.1]{MR3642734}), 
	and because of the regularity of classical solutions
	(see Theorem~\ref{classical-solutions}) and Remark~\ref{flap-prop}, we have
	\begin{equation*}
            \dt E|u_n-u_m|^p = -E\fLap (u_n-u_m)g(u_n-u_m) \leq -E\fLap g(u_n-u_m)=0.
	\end{equation*}
        It follows that
        \begin{equation*}
            E|u_n(t)-u_m(t)|^p \leq E|u_0^n-u_0^m|^p
        \end{equation*}
        for every $t\geq 0$. However, 
        \begin{equation*}
            \lim_{n,m\to\infty} E|u_0^n-u_0^m| = 0,
        \end{equation*}
        thus 
        \begin{equation*}
         \lim_{n,m\to\infty} \sup_{t\geq0} E|u_n(t)-u_m(t)|^p = 0,
        \end{equation*}
        which shows that $u^n$ is a Cauchy sequence, and therefore converges, in $\JIso_{K,p}$.
    \end{proof}
    \begin{theorem}\label{main}
        There exists 
        a solution $u$ to problem~\eqref{polyburgers} such that $E|u(t)|^p\leq E|u_0|^p$ for every $t\geq0$ and $p\geq2$.
    \end{theorem}
    \begin{proof}
        Let $u$ be the limit of the approximative solutions $u^n$ given by Lemma~\ref{l1}.
         The estimates 
         \begin{equation}\label{yet-another-estimate}
          E|u(t)|^p\leq E|u_0|^p
         \end{equation}
         follow from the analogous properties of the
         approximative solutions given in Theorem~\ref{rosenblatt-lp}.
          Consider
          \begin{equation*}
           F(u) = P_{t}u(0) + \int_0^{t} \grad P_{t-\tau} f(u(\tau))\,d\tau.
          \end{equation*}
          By Lemma~\ref{semigroup-estimate} we have
          \begin{multline*}
              \|u^n(t)-F(u)(t)\|_2 \\\leq \|P_{t}(u^n_0-u_0)\|_2 + c_s\int_0^{t} (t-\tau)^{-\frac{1}{2s}} \|f(u^n(\tau))-f(u(\tau))\|_2\,d\tau.
	  \end{multline*}
	  Because we assume $|f(x)-f(y)|\leq C|x-y|\big(|x|^q+|y|^q\big)$, then
          \begin{multline*}
              \|u^n(t)-F(u)(t)\|_2  - \|P_{t}(u^n_0-u_0)\|_2 \\\leq
              c_sC\int_0^{t} (t-\tau)^{-\frac{1}{2s}} \|u^n(\tau)-u(\tau)\|_2\big(\|u^n(\tau)\|_2^q+\|u(\tau)\|_2^q\big)\,d\tau.
          \end{multline*}
         By Theorem~\ref{rosenblatt-lp} and estimate~\eqref{yet-another-estimate} we get
         \begin{multline*}
          \int_0^{t} (t-\tau)^{-\frac{1}{2s}} \|u^n(\tau)-u(\tau)\|_2\big(\|u^n(\tau)\|^q_2+\|u(\tau)\|^q_2\big)\,d\tau\\
          \leq 2\|u_0\|_2^q\int_0^{t} (t-\tau)^{-\frac{1}{2s}} \|u^n(\tau)-u(\tau)\|_2\,d\tau.
         \end{multline*}
         For every $K>0$ (cf. estimates~\eqref{norm-est1},~\eqref{gamma}) we have
         \begin{multline*}
          \sup_{t\geq0} e^{-tK}\int_0^{t} (t-\tau)^{-\frac{1}{2s}} \|u^n(\tau)-u(\tau)\|_2\,d\tau\\\leq
          \Big(\sup_{0<\tau<t}e^{-\tau K}\|u^n(\tau)-u(\tau)\|_2\Big) \int_0^{t} (t-\tau)^{-1/2s} e^{-(t-\tau) K}\,d\tau
          \\\leq  K^{-1+\frac{1}{2s}}\Gamma(1-\tfrac{1}{2s})\K{u^n-u}{2}.
         \end{multline*}
         Finally we obtain
         \begin{equation*}
         \K{u^n-F(u)}{2}\leq c_sCK^{-1+\frac{1}{2s}}\Gamma(1-\tfrac{1}{2s})\K{u^n-u}{2}
         \end{equation*}
         and by Lemma~\ref{l1}
         \begin{equation*}
         \lim_{n\to\infty}\K{u^n-F(u)}{2}=0.
         \end{equation*}
         This of course means that $u=F(u)$ (Lemma~\ref{l1} again) and $u$ is a solution.
      \end{proof}
      \begin{remark}
       As in Remark~\ref{uniqueness-remark}, the solution we have constructed in this section
       may not be unique in the context of Definition~\ref{polydefinition}.
      \end{remark}
    \end{subsection}
    \section*{Acknowledgements}
    The author wishes to thank Grzegorz Karch for his invaluable help in preparing this paper.
    This work was partially supported by the Polish NCN grant No.~2016/23/B/ST1/00434.
\end{section}
\bibliographystyle{amsalpha}
  \bibliography{nonlinear3}

\providecommand{\bysame}{\leavevmode\hbox to3em{\hrulefill}\thinspace}
\providecommand{\MR}{\relax\ifhmode\unskip\space\fi MR }
\providecommand{\MRhref}[2]{%
  \href{http://www.ams.org/mathscinet-getitem?mr=#1}{#2}
}
\providecommand{\href}[2]{#2}
\begin{thebibliography}{BKW01b}

\bibitem[AL01]{MR1859007}
V.~V. Anh and N.~N. Leonenko, \emph{Spectral analysis of fractional kinetic
  equations with random data}, J. Statist. Phys. \textbf{104} (2001), no.~5-6,
  1349--1387.

\bibitem[AL02]{MR1939652}
\bysame, \emph{Renormalization and homogenization of fractional diffusion
  equations with random data}, Probab. Theory Related Fields \textbf{124}
  (2002), no.~3, 381--408.

\bibitem[AMS94]{MR1305783}
S.~Albeverio, S.~A. Molchanov, and D.~Surgailis, \emph{Stratified structure of
  the {U}niverse and {B}urgers' equation---a probabilistic approach}, Probab.
  Theory Related Fields \textbf{100} (1994), no.~4, 457--484.

\bibitem[Bak01]{MR1978661}
Yu.~Yu. Bakhtin, \emph{A functional central limit theorem for transformed
  solutions of the multidimensional {B}urgers equation with random initial
  data}, Teor. Veroyatnost. i Primenen. \textbf{46} (2001), no.~3, 427--448.
  \MR{1978661}

\bibitem[Bil95]{MR1324786}
P.~Billingsley, \emph{Probability and measure}, third ed., Wiley Series in
  Probability and Mathematical Statistics, John Wiley \& Sons, Inc., New York,
  1995, A Wiley-Interscience Publication.

\bibitem[BKW99]{MR1708995}
P.~Biler, G.~Karch, and W.~A. Woyczy\'nski, \emph{Asymptotics for multifractal
  conservation laws}, Studia Math. \textbf{135} (1999), no.~3, 231--252.

\bibitem[BKW01a]{MR1881259}
\bysame, \emph{Asymptotics for conservation laws involving {L}\'evy diffusion
  generators}, Studia Math. \textbf{148} (2001), no.~2, 171--192.

\bibitem[BKW01b]{MR1849690}
\bysame, \emph{Critical nonlinearity exponent and self-similar asymptotics for
  {L}\'evy conservation laws}, Ann. Inst. H. Poincar\'e Anal. Non Lin\'eaire
  \textbf{18} (2001), no.~5, 613--637.

\bibitem[BPSV14]{MR3211862}
B.~Barrios, I.~Peral, F.~Soria, and E.~Valdinoci, \emph{A {W}idder's type
  theorem for the heat equation with nonlocal diffusion}, Arch. Ration. Mech.
  Anal. \textbf{213} (2014), no.~2, 629--650.

\bibitem[CS17]{MR3642734}
L.~A. Caffarelli and Y.~Sire, \emph{On some pointwise inequalities involving
  nonlocal operators}, Harmonic analysis, partial differential equations and
  applications, Appl. Numer. Harmon. Anal., Birkh\"auser/Springer, Cham, 2017,
  pp.~1--18.

\bibitem[DGV03]{MR2019032}
J.~Droniou, T.~Gallouet, and J.~Vovelle, \emph{Global solution and smoothing
  effect for a non-local regularization of a hyperbolic equation}, J. Evol.
  Equ. \textbf{3} (2003), no.~3, 499--521, Dedicated to Philippe B\'enilan.

\bibitem[GS69]{MR0247660}
I.~I. Gikhman and A.~V. Skorokhod, \emph{Introduction to the theory of random
  processes}, Translated from the Russian by Scripta Technica, Inc, W. B.
  Saunders Co., Philadelphia, Pa., 1969.

\bibitem[GV64]{MR0173945}
I.~M. Gel'fand and N.~Ya. Vilenkin, \emph{Generalized functions. {V}ol. 4:
  {A}pplications of harmonic analysis}, Translated by Amiel Feinstein, Academic
  Press, New York, 1964.

\bibitem[IL89]{MR1009786}
A.~V. Ivanov and N.~N. Leonenko, \emph{Statistical analysis of random fields},
  Mathematics and its Applications (Soviet Series), vol.~28, Kluwer Academic
  Publishers Group, Dordrecht, 1989, With a preface by A. V. Skorokhod,
  Translated from the Russian by A. I. Kochubinski\u\i.

\bibitem[IS17]{2017arXiv170302908I}
L.~{Ignat} and D.~{Stan}, \emph{{Asymptotic behaviour for fractional
  diffusion-convection equations}}, ArXiv e-prints (2017).

\bibitem[JW01]{MannJr.2001}
J.~A.~Mann Jr. and W.~A. Woyczy\'nski, \emph{Growing fractal interfaces in the
  presence of self-similar hopping surface diffusion}, Physica A: Statistical
  Mechanics and its Applications \textbf{291} (2001), no.~1-4, 159--183.

\bibitem[Kru17]{MR3628179}
M.~Krupski, \emph{Scaling limits of solutions of linear evolution equations
  with random initial conditions}, Stoch. Dyn. \textbf{17} (2017), no.~3,
  1750019, 24.

\bibitem[KW08]{MR2373320}
G.~Karch and W.~A. Woyczy\'nski, \emph{Fractal {H}amilton-{J}acobi-{KPZ}
  equations}, Trans. Amer. Math. Soc. \textbf{360} (2008), no.~5, 2423--2442.

\bibitem[Leo99]{MR1687092}
N.~N. Leonenko, \emph{Limit theorems for random fields with singular spectrum},
  Mathematics and its Applications, vol. 465, Kluwer Academic Publishers,
  Dordrecht, 1999.

\bibitem[LS96]{MR1409835}
V.~A. Liskevich and Yu.~A. Semenov, \emph{Some problems on {M}arkov
  semigroups}, Schr\"odinger operators, {M}arkov semigroups, wavelet analysis,
  operator algebras, Math. Top., vol.~11, Akademie Verlag, Berlin, 1996,
  pp.~163--217.

\bibitem[LW98]{MR1642664}
N.~N. Leonenko and W.~A. Woyczy\'nski, \emph{Exact parabolic asymptotics for
  singular {$n$}-{D} {B}urgers' random fields: {G}aussian approximation},
  Stochastic Process. Appl. \textbf{76} (1998), no.~2, 141--165.

\bibitem[Rao12]{MR2840012}
M.~M. Rao, \emph{Random and vector measures}, Series on Multivariate Analysis,
  vol.~9, World Scientific Publishing Co. Pte. Ltd., Hackensack, NJ, 2012.

\bibitem[Ros68]{MR0264252}
M.~Rosenblatt, \emph{Remarks on the {B}urgers equation}, J. Mathematical Phys.
  \textbf{9} (1968), 1129--1136.

\bibitem[VSCC92]{MR1218884}
N.~Th. Varopoulos, L.~Saloff-Coste, and T.~Coulhon, \emph{Analysis and geometry
  on groups}, Cambridge Tracts in Mathematics, vol. 100, Cambridge University
  Press, Cambridge, 1992.

\bibitem[Woy98]{MR1732301}
W.~A. Woyczy\'nski, \emph{Burgers-{KPZ} turbulence}, Lecture Notes in
  Mathematics, vol. 1700, Springer-Verlag, Berlin, 1998, G{\"o}ttingen
  lectures.

\end{thebibliography}
\end{document}